\documentclass{amsart}
\usepackage{amssymb}
\usepackage{amsthm}
\usepackage{epic}
\usepackage{amsmath}
\usepackage{amsfonts}

\newtheorem{theorem}{Theorem}
\newtheorem{proposition}{Proposition}

\usepackage[utf8]{inputenc}

\def\bigO{\mathcal{O}}

\def\rk#1{\left\lceil \frac{R-#1}{2}\right\rceil}

\def\bigO{\mathcal{O}}

\title[Stability of approximate Taylor methods for ODE]{On the stability of Approximate Taylor methods for ODE and
  their relationship with Runge-Kutta schemes} 
\author{A.~Baeza}
\thanks{Departament de Matemàtiques,  Universitat de
  València. Av. Vicent Andr\'es Estell\'es, 46100 Burjassot  (Spain); email: antonio.baeza@uv.es
}
\author{S.~Boscarino}
\thanks{Department of Mathematics and Computer Science, University of Catania, 
 95125 Catania (Italy); email: boscarino@dmi.unict.it
}
\author{P.~Mulet}
\thanks{Departament de Matemàtiques,  Universitat de
  València. Av. Vicent Andr\'es Estell\'es, 46100 Burjassot  (Spain);
  email:  pep.mulet@uv.es
}
\author{G.~Russo}
\thanks{
Department of Mathematics and Computer Science, University of Catania, 
 95125 Catania (Italy); email:  russo@dmi.unict.it
}
\author{D.~Zorío}
\thanks{
Centro de Investigaci\'on en Ingenier\'{\i}a Matem\'atica, Universidad
de Concepci\'on, Casilla 160C Concepci\'on (Chile); email: dzorio@ci2ma.udec.cl
}
\date{}

\begin{document}

\begin{abstract}
  In [Baeza et al., Computers and Fluids, \textbf{159}, 156--166 (2017)]
  a new method for the numerical solution of ODEs is presented.
  This methods can be regarded as an approximate formulation of the
  Taylor methods and 
  it   follows an approach 
that
  has  a much easier implementation than the
  original Taylor methods, since only the functions in the ODEs, and
  not their high order derivatives, are needed. In this reference, the
  absolute stability region of the new methods is conjectured to be coincident
  with that of their exact counterparts. There is also a conjecture
  about their relationship with Runge-Kutta methods. In this work we
  answer positively both conjectures.

\keywords{
ODE integrators, Taylor methods
}
\end{abstract}
\maketitle

\section{Introduction}
In \cite{BaezaBoscarinoMuletEtAl2017}  we develop a numerical scheme
consisting on an 
approximate formulation of  Taylor methods for ODEs, akin to the
scheme proposed by Zorío et. al. in  
\cite{ZorioBaezaMulet17} for hyperbolic conservation laws.

In \cite{BaezaBoscarinoMuletEtAl2017} we show that high order schemes can be
achieved with a computational cost that is competitive with that of
Taylor methods, with the advantage that the approximate Taylor methods
do not require the knowledge of higher order derivatives of the
function in the ODE. In the cited reference there were the conjectures
about the absolute stability region of the approximate methods
coinciding with that of their exact counterparts and that they can be
cast as Runge-Kutta schemes.   We give here positive   answers to
both conjectures. 

This paper is organized as follows:  in Section \ref{cka}  we recall 
 the  approximate Taylor methods  and analyze their  stability; in
 Section \ref{rk} we 
prove that the approximate Taylor methods are indeed Runge-Kutta
methods.

\section{ Approximate Taylor methods}\label{cka}

Without loss of generality, we consider initial value problems for
systems of $m$ autonomous  ODEs for
the unknown $u=u(t)$:
\begin{equation}\label{eq:hcl}
    u'=f(u), u(0)=u_0,
\end{equation}
 for nonautonomous systems can be cast as autonomous systems by
considering $t$ as a new unknown and by correspondingly inserting a
new equation $t'=1$ and initial condition $t(0)=0$.

We aim to obtain approximations   $v_{h,n}\approx u(t_n) $ of the
solution $u$ of \eqref{eq:hcl} on  $t_n=nh $, 
$n=0,\dots$,  with spacing $h>0$.
We use the following notation for time derivatives of $u$:
   \begin{align*}
     u_{h,n}^{(l)}&=\frac{d^{l} u}{d t^l}(t_n).
   \end{align*}

Our goal is to obtain an $R$-th order accurate numerical scheme, i.e., a scheme
with a local truncation error 
of order $R+1$, based on the Taylor expansion  of the solution $u$
from time $t_n$ to the next time $t_{n+1}$:
$$u_{h,n+1}^{(0)}=u(t_{n+1})=\sum_{l=0}^R\frac{h^l}{l!}u_{h,n}^{(l)}+\bigO(h^{R+1}).$$
To achieve this we aim to define  approximations
$v_{h,n}^{(l)}\approx u_{h,n}^{(l)}$, with $v^{(0)}_{h,n}=v_{h,n}$,  by recursion on $l$, such that 
\begin{align}\label{eq:200}
  v_{h,n}^{(0)}=u_{h,n}^{(0)}=u(t_n) \Rightarrow
  v_{h,n}^{(l)}=u_{h,n}^{(l)}+\bigO(h^{R+1-l}), l=1,\dots,R.
\end{align}

Let us first briefly review how Taylor methods for \eqref{eq:hcl} are obtained,
by a simple scalar ($m=1$) example for third order accuracy (see, for instance,
\cite{HairerNorsettWanner93} for more details). 
 The fact that $u$ solves \eqref{eq:hcl} implies that
 \begin{align*}
   u''&=f(u)'=f'(u)u'\\
   u'''&=f(u)''=f''(u)(u')^2+f'(u)u'',
 \end{align*}
 which yields 
 \begin{equation}\label{eq:40}
  \begin{aligned}
u_{h,n}^{(1)}&=u'(t_n)=f(u_n^{(0)})\\
u_{h,n}^{(2)}&=u''(t_n)=f'(u_n^{(0)})u_n^{(1)}\\
u_{h,n}^{(3)}&=u'''(t_n)=f''(u_n^{(0)})(u_{h,n}^{(1)})^2+f'(u)u_{h,n}^{(2)}.
\end{aligned}
\end{equation}
We have just proven that the definition 
 \begin{equation*}
  \begin{aligned}
v_{h,n}^{(1)}&=f(v_{h,n}^{(0)})\\
v_{h,n}^{(2)}&=f'(v_{h,n}^{(0)})v_{h,n}^{(1)}\\
v_{h,n}^{(3)}&=f''(v_{h,n}^{(0)})(v_{h,n}^{(1)})^2+f'(u)(v_{h,n}^{(2)}),
\end{aligned}
\end{equation*}
satisfies \eqref{eq:200} with no error and that 
 \begin{equation*}
  \begin{aligned}
v_{h,n+1}&=v_{h,n}+hv_{h,n}^{(1)}+\frac{h^2}{2}v_{h,n}^{(2)}+\frac{h^3}{6}v_{h,n}^{(3)}
\end{aligned}
\end{equation*}
yields a third order scheme, the \textit{exact} third order Taylor scheme.

It is readily seen  that the $R$-th order Taylor methods applied to
the equation $u'=\lambda u$ is given by
\begin{equation}\label{eq:QTaylor}
  v_{h,n+1}=Q(h\lambda) v_{h,n}, \quad
  Q(x)=1+x+\frac{x^2}{2}+\dots+\frac{x^R}{R!}, 
\end{equation}
and, correspondingly,  their region of absolute stability is
\begin{equation*}
  \{ z\in\mathbb C \colon |Q(z)|\leq 1 \}.
\end{equation*}

We propose in \cite{BaezaBoscarinoMuletEtAl2017} 
 an alternative solver, which is much less expensive for large
$m, R$
and agnostic about the equation, in the sense that its only
requirement is the knowledge of the function $f$.
 These solvers are based on the observation that  approximations of
 \eqref{eq:40} 
can be easily obtained by using finite differences of enough order,
rather than using those  expressions, that  would only be exact for the
analysis of the local truncation error.

For the sake of completeness we briefly describe here the scheme
proposed in \cite{BaezaBoscarinoMuletEtAl2017}, along with  the following
auxiliary notation:
For a function $u\colon \mathbb R\to \mathbb R^{m}$, we denote its
sampling  on  the grid defined by a base
  point $a$ and grid space $h$ by
  \begin{equation*}
    G_{a,h}(u)\colon\mathbb{Z} \to \mathbb {R}^{m},\quad
    G_{a,h}(u)_i=u(a+ih).
  \end{equation*}
For naturals $p,q$, we denote by  $ \Delta^{p,q}_{h}$
the centered finite differences operator that  approximates $p$-th order
  derivatives to order $2q$ on grids with spacing $h$, which, for any $u$
  sufficiently differentiable,  satisfies (see \cite[Proposition
  1]{ZorioBaezaMulet17} for the  details):
  \begin{equation}\label{eq:3}
    \left|\Delta^{p,q}_{h}       G_{a,h}(u)-u^{(p)}(a)\right|\leq
      K_{p,q}\max\{ |u^{(p+2q)}(x)|\colon |x-a|\leq r_{p,q}h\}h^{2q},
  \end{equation}
  for some $r_{p,q}\in\mathbb N$.
  
  Given $v_{h,n}$, 
the approximations $v^{(k)}_{h,n} \approx
u^{(k)}_{h,n}$ are defined by recursion on $k=0,\dots,R$ as follows:
\begin{equation}\label{eq:15}
  \begin{aligned}
    v^{(0)}_{h,n}&=v_{h,n}\\
    v^{(1)}_{h,n}&=f(v_{h,n})\\
    v^{(k+1)}_{h,n} &=
    \Delta_{h}^{k,\left\lceil \frac{R-k}{2}\right\rceil}\Big(G_{0,h}\big(f(T_{h,n}^{k})\big)\Big),\\
  \end{aligned}
\end{equation}   
where $T_{h,n}^{k}$ is the $k$-th degree approximate Taylor polynomial given by
\begin{align}\label{eq:300}
  T_{h,n}^{k}(\rho)=\sum_{l=0}^{k}\frac{v^{(l)}_{h,n} }{l!} \rho^l.
\end{align}

With all this notation, the proposed scheme is:
\begin{equation}\label{eq:60}
  v_{h,n+1}=T_{h,n}^{R}(h)=\sum_{l=0}^{R}\frac{v_{h,n}^{(l)}}{l!}h^{l}=v_{h,n}+h\sum_{l=1}^{R}w_{h,n}^{(l)},\quad
  w_{h,n}^{(l)}:=\frac{v_{h,n}^{(l)}}{l!}h^{l-1}.
\end{equation}

The following result is proven in \cite{BaezaBoscarinoMuletEtAl2017}
and it is a simplified adaptation of the corresponding
result in \cite{ZorioBaezaMulet17}.

\begin{theorem}\label{th:1}
  The scheme defined by \eqref{eq:15} and \eqref{eq:60} is $R$-th
  order accurate. 
\end{theorem}

The following result, which corresponds to \eqref{eq:QTaylor} and  
has been established for orders $R=2,3,4$ 
in \cite{BaezaBoscarinoMuletEtAl2017}, is next proven for any $R$.

\begin{proposition}
   When applied to
  homogeneous linear   systems $u'=Au$, for a $m\times m$ matrix $A$, 
  the scheme defined by \eqref{eq:15} and \eqref{eq:60} coincides with
  the exact Taylor method of the same order  and therefore has the same
  stability region.
\end{proposition}
  \begin{proof}
We consider $f(u)=Au$, for an $m\times m$ matrix $A$, and first
prove  by induction on $k\leq R$  that
\begin{equation}\label{eq:800}
  v_{h,n}^{(k)}=A^k v_{h,n},
\end{equation}
the case
$k=0$ being clear from \eqref{eq:15}. Assume 
$k+1\leq R$, 
$v_{h,n}^{(k)}=A^k v_{h,n}$ and use \eqref{eq:15} and the linearity
of $    \Delta_{h}^{k,q}$ to write
\begin{align}\label{eq:810}
  v^{(k+1)}_{h,n} &=
    \Delta_{h}^{k,q}\left(G_{0,h}\big(f(T_{h,n}^{k})\big)\right)=A
    \Delta_{h}^{k,q}\left(G_{0,h}\big(T_{h,n}^{k}\big)\right),
  \end{align}
  for $q=\left\lceil
        \frac{R-k}{2}\right\rceil$.
  Now \eqref{eq:3}, the fact that $T_{h,n}^{k}$ is the $k$-th degree
  polynomial given in \eqref{eq:300} and the induction hypothesis yield
  \begin{align*}
    \Delta^{k,q}_{h}\left(
    G_{0,h}(T_{h,n}^{k})\right)&=(T_{h,n}^{k})^{(k)}(0)=v_{h,n}^{(k)}=A^{k}v_{h,n},    
  \end{align*}
  which, together with \eqref{eq:810}, therefore yields \eqref{eq:800}
  for $k+1$, thus concluding the  proof by induction of \eqref{eq:800}.

  Now, \eqref{eq:60} immmediately gives  that $v_{h,n+1}=Q(hA)v_{h,n}$, where $Q$
  is the $R$-degree polynomial in  \eqref{eq:QTaylor}.
\qed  
\end{proof}  

\section{Relationship with Runge-Kutta schemes}
\label{rk}
In \cite{BaezaBoscarinoMuletEtAl2017}
it was  conjectured that the $R$-th- order approximate Taylor method  could be
cast as a Runge-Kutta method with $R^2+\bigO(R)$ stages, which  is
asymtotically larger than the upper bound  on the number of stages to
achieve $R$-th order $\frac{3}{8}R^2+\bigO(R)$
 given in \cite{Butcher2008}.

In this section we prove that the approximate Taylor methods are
indeed Runge-Kutta schemes, by properly identifying their
corresponding Butcher arrays.
The following result, which is taken from \cite{ZorioBaezaMulet17}, is at
the foundation of this identification.

\begin{theorem}\label{th:1}
  For any  $p, q\in\mathbb{N}$, there exist  $\beta_{l}^{p,q}$,
  $l=0,\dots,s:=\lfloor \frac{p-1}{2}\rfloor+q$ such that
  \begin{align}\label{eq:1003bis}
    \Delta^{p,q}_{h} v=\frac{1}{h^p}\sum_{l=0}^{s}\beta_{l}^{p,q}(v_{h,l}+(-1)^pv_{h,-l})
  \end{align}
  satisfies \eqref{eq:3}.
\end{theorem}

With the notation:
  \begin{align}\label{eq:4}
    f_{k,j}&=f(T_{h,n}^{k}(jh))=G_{0,h}\big(f(T_{h,n}^{k})\big)_j
\end{align}
and \eqref{eq:60} and \eqref{eq:1003bis},
equation \eqref{eq:15} yields:
\begin{equation}\label{eq:5}
\begin{aligned}
w_{h,n}^{(l+1)}=\frac{h^{l}}{(l+1)!}v^{(l+1)}_{n} &=
\frac{1}{(l+1)!}
\sum_{i=-m_{l,R}}^{m_{l,R}} \gamma_{i}^{l,R}f_{l,i},\\
&m_{l,R}=s_{l,\rk{l}},
\gamma_{i}^{l,R}= (sign(i))^l\beta_{i}^{l,\rk{l}}.
\end{aligned}
\end{equation}
This equation is also valid for $l=0$ if one takes $m_{0,R}=0$ and
$\gamma_{0}^{0,R}=1$:

\begin{equation*}
  w_{h,n}^{(1)}=v_{h,n}^{(1)}=f(v_{h,n})=f_{0,0}.
\end{equation*}  

From \eqref{eq:4} and \eqref{eq:5}:
\begin{align}
\notag
  f_{k,j}&=f(T_{h,n}^{k}(jh))=f(v_{h,n}+h\sum_{l=0}^{k-1}w_{h,n}^{(l+1)}j^{l+1})\\
  \notag
    &=f\big(v_{h,n}+h\sum_{l=0}^{k-1}(\frac{1}{(l+1)!}
\sum_{i=-m_{l,R}}^{m_{l,R}} \gamma_{i}^{l,R}f_{l,i})j^{l+1}\big)\\
\notag
f_{k,j}&=f\big(v_{h,n}+h\sum_{l=0}^{k-1} 
\sum_{i=-m_{l,R}}^{m_{l,R}}\frac{j^{l+1}\gamma_{i}^{l,R}}{(l+1)!} f_{l,i}\big).
\end{align}
Notice that $f_{k,0}=f_{0,0}=f(v_{h,n})$ $\forall k$.
Therefore
\begin{align}\label{eq:6}
  f_{k,j}&=f\big(v_{h,n}+h\sum_{l=0}^{k-1} 
\frac{j^{l+1}\gamma_{0}^{l,R}}{(l+1)!} f_{0,0}+h\sum_{l=0}^{k-1} 
\sum_{i=-m_{l,R}, i\neq 0}^{m_{l,R}}\frac{j^{l+1}\gamma_{i}^{l,R}}{(l+1)!} f_{l,i}\big).
\end{align}

Similarly, from \eqref{eq:5}, equation \eqref{eq:60} reads:
\begin{align}
\label{eq:7}
  v_{h,n+1}=v_{h,n}+h\sum_{l=0}^{R-1} 
\frac{j^{l+1}\gamma_{0}^{l,R}}{(l+1)!} f_{0,0}+h\sum_{l=0}^{R-1} 
\sum_{i=-m_{l,R}, i\neq 0}^{m_{l,R}}\frac{j^{l+1}\gamma_{i}^{l,R}}{(l+1)!} f_{l,i}
\end{align}

Let us consider the subset of $\mathbb Z^2$:
\begin{align*}
  \mathcal{S}_{R}=
  \{(0,0)\}\cup
\Big\{(l, i) / l\in \{1,R-1 \}, i\in \{-m_{l,R}, \dots,
m_{l,R} \}\setminus\{0\}\Big\},
\end{align*}
with size 
\begin{equation}
  \label{eq:8}
n_R:=|\mathcal{S}_{R}|=1+2\sum_{l=1}^{R-1} m_{l,R} 
\end{equation}
and with the lexicographical ordering given by the bijection
\begin{equation*}
  I_R\colon \mathcal{S}_R \to \{1,\dots,n_R \},\quad
  I_{R}(l, i)
  =
  \begin{cases}
    0& l=0\\
    2+2\sum_{k<l} m_{k,R}+i+m_{l,R} & 0<l< R, i<0\\
    1+2\sum_{k<l} m_{k,R}+i+m_{l,R} & 0<l< R, i>0.
  \end{cases}    
\end{equation*}
It can be proven that for $l=1,\dots,R-1$:
\begin{align*}
  m_{l,R}=
  \begin{cases}
    \frac{R-1}{2} & \text{$R$ is odd}\\
      \frac{R}{2}-1 & \text{$R$ is even, $l$ is even}\\
      \frac{R}{2} & \text{$R$ is even, $l$ is odd}
  \end{cases}    ,
\end{align*}
and from here and \eqref{eq:8} that
\begin{align*}
  n_R=
  \begin{cases}
    1+(R-1)^2 & \text{$R$ is odd}\\
    2+(R-1)^2 & \text{$R$ is even}
  \end{cases}    
\end{align*}

We define the $n_R\times n_R$ matrix $A^{(R)}$ whose nonzero entries
are given by:
\begin{align}\label{eq:1000}
  A^{(R)}_{I_R(k,j), I_R(0, 0)}&=  \sum_{l=0}^{k-1} 
\frac{j^{l+1}\gamma_{0}^{l,R}}{(l+1)!} &
A^{(R)}_{I_R(k,j), I_R(l,i)}&=\frac{j^{l+1}\gamma_{i}^{l,R}}{(l+1)!},
\end{align}
for $k=1,\dots,R-1$, $j=-m_{k,R},\dots,m_{k,R}, j\neq 0$ and
$l=1,\dots,k-1$, $i=-m_{l,R},\dots,m_{l,R}, i\neq 0$. 

We also define the $1\times n_R$ vector $b^{(R)}$ as follows:
\begin{align}\label{eq:1100}
b^{(R)}_{I_R(0,0)}&=  \sum_{l=0}^{R-1} 
\frac{\gamma_{0}^{l,R}}{(l+1)!} &  b^{(R)}_{I_R(l,
  i)}&=\frac{\gamma_{i}^{l,R}}{(l+1)!}, l>0,
\end{align}
for $l=1,\dots,R-1$, $i=-m_{l,R},\dots,m_{l,R}, i\neq 0$.

With this notation we have the following

\begin{theorem}\label{th:2}
  \leavevmode
  \begin{enumerate}
  \item \label{it:0}   The scheme defined by \eqref{eq:15} and
    \eqref{eq:60} is  a $n_R$-stages Runge-Kutta scheme with Butcher
    array given by 
\begin{align*}
  \begin{array}{c|c}
    c^{(R)}&A^{(R)}\\    \hline
    &b^{(R)}
  \end{array},    
\end{align*}
where $c^{(R)}_{i}=\sum_{j=1}^{i-1} A^{(R)}_{i,j}$.
    
  \item \label{it:1}   $A^{(R)}$ is a block strictly lower $R\times R$ triangular
    matrix, with blocks of sizes $1, m_{1,R},\dots,m_{R-1,R}$.
  \item \label{it:2}$(A^{(R)})^R=\mathbf{0}$.
  \item \label{it:3} $\text{rank} A^{(R)}=R$.
  \end{enumerate}      
\end{theorem}

\begin{proof}
  Item \ref{it:0} follows from \eqref{eq:6}, \eqref{eq:7},
  \eqref{eq:1000} and  \eqref{eq:1100}.
  
  Item \ref{it:1} is clear from \eqref{eq:1000} and item \ref{it:2}
  from this one. From \eqref{eq:1000} 
  \begin{equation*}
    A^{(R)}_{I_R(1,1), 1}=      \gamma_{0}^{0,R}=1,
  \end{equation*}
  therefore $A^{(R)}(:, 1)\neq \mathbf{0}$.

  From the proof of Theorem \ref{th:1}, for any  $1\leq l<R-1$
  there exists $i\in\{-m_{l,R}, \dots, m_{l,R}\}\setminus \{0\}$ such
  that $\gamma_{i}^{l,R}\neq 0$.
  Therefore, for any $1\leq k<R$,
  and any $j\in\{-m_{k,R}, \dots, m_{k,R}\}\setminus \{0\}$ and any $1\leq l<k$
  there exists $i_l\in\{-m_{l,R}, \dots, m_{l,R}\}\setminus \{0\}$ such
  that $\gamma_{i_l}^{l,R}\neq 0$. 

From \eqref{eq:1000}, $A^{(R)}_{I_R(k,j), I_R(l,i_l)}\neq 0$ and 
\begin{align*}
A^{(R)}_{I_R(k,j), I_R(l,i)}&=A^{(R)}_{I_R(k,j), I_R(l,i_l)}
\frac{\gamma_{i}^{l,R}}{\gamma_{i_l}^{l,R}}
\end{align*}
We have just proven that the columns $A^{(R)}(:, I_R(l,i))$,
$i=-m_{l,R},\dots,m_{l,R}$, $i\neq 0$, are
proportional to $A^{(R)}(:, I_R(l,i_l))\neq 0$.

By the structure of
the matrix, $$A^{(R)}(:, 1),A^{(R)}(:, I_R(1,i_1)), \dots, A^{(R)}(:,
I_R(R-1,i_{R-1}))$$ are linearly independent and all the columns of the
matrix are proportional to one of these. Therefore, $\text{rank}A^{(R)}=R$.
\end{proof}

\section*{Acknowledgments}
 Antonio Baeza, Pep Mulet and David Zor\'{\i}o are 
  supported by Spanish MINECO grants MTM 2014-54388-P and
  MTM2017-83942-P. D. Zor\'{\i}o is also 
  supported by Fondecyt project 3170077. Giovanni Russo and Sebastiano Boscarino have been partially supported by Italian PRIN 2009 project ``Innovative numerical methods for hyperbolic problems with application to fluid dynamics, kinetic theory, and computational biology'', Prot. No. 2009588FHJ.

\end{document}